\documentclass[oneside,english]{amsart}
\usepackage[T1]{fontenc}
\usepackage[latin9]{inputenc}
\usepackage{geometry}
\usepackage{amsbsy}
\usepackage{amstext}
\usepackage{amsthm}
\usepackage{amssymb}

\makeatletter
\numberwithin{equation}{section}
\numberwithin{figure}{section}
\theoremstyle{plain}
\newtheorem{thm}{\protect\theoremname}[section]
  \theoremstyle{definition}
  \newtheorem{defn}[thm]{\protect\definitionname}
  \theoremstyle{definition}
  \newtheorem{problem}[thm]{\protect\problemname}
  \theoremstyle{definition}
  \newtheorem{example}[thm]{\protect\examplename}
  \theoremstyle{plain}
  \newtheorem{conjecture}[thm]{\protect\conjecturename}
  \theoremstyle{plain}
  \newtheorem{lem}[thm]{\protect\lemmaname}
  \theoremstyle{plain}
  \newtheorem{cor}[thm]{\protect\corollaryname}
  \theoremstyle{plain}
  \newtheorem{prop}[thm]{\protect\propositionname}
  \theoremstyle{remark}
  \newtheorem{claim}[thm]{\protect\claimname}

\makeatother

\usepackage{babel}
  \providecommand{\claimname}{Claim}
  \providecommand{\conjecturename}{Conjecture}
  \providecommand{\corollaryname}{Corollary}
  \providecommand{\definitionname}{Definition}
  \providecommand{\examplename}{Example}
  \providecommand{\lemmaname}{Lemma}
  \providecommand{\problemname}{Problem}
  \providecommand{\propositionname}{Proposition}
\providecommand{\theoremname}{Theorem}

\begin{document}
\global\long\def\f{\mathcal{F}}
\global\long\def\p{\mathcal{P}}
\global\long\def\j{\mathcal{J}}
\global\long\def\g{\mathcal{G}}

\title{On set systems without a simplex-cluster and the Junta method}

\author{Noam Lifshitz }
\begin{abstract}
A family $\{A_{0},\ldots,A_{d}\}$ of $k$-element subsets of $[n]=\{1,2,\ldots,n\}$
is called a simplex-cluster if $A_{0}\cap\cdots\cap A_{d}=\varnothing$,
$|A_{0}\cup\cdots\cup A_{d}|\le2k$, and the intersection of any $d$
of the sets in $\{A_{0},\ldots,A_{d}\}$ is nonempty. In 2006, Keevash
and Mubayi conjectured that for any $d+1\le k\le\frac{d}{d+1}n$,
the largest family of $k$-element subsets of $[n]$ that does not
contain a simplex-cluster is the family of all $k$-subsets that contain
a given element. We prove the conjecture for all $k\ge\zeta n$ for
an arbitrarily small $\zeta>0$, provided that $n\ge n_{0}(\zeta,d)$.

We call a family $\{A_{0},\ldots,A_{d}\}$ of $k$-element subsets
of $[n]$ a $(d,k,s)$-cluster if $A_{0}\cap\cdots\cap A_{d}=\varnothing$
and $|A_{0}\cup\cdots\cup A_{d}|\le s$. We also show that for any
$\zeta n\le k\le\frac{d}{d+1}n$ the largest family of $k$-element
subsets of $\left[n\right]$ that does not contain a $(d,k,(\frac{d+1}{d}+\zeta)k)$-cluster
is again the family of all $k$-subsets that contain a given element,
provided that $n\ge n_{0}(\zeta,d)$. 

Our proof is based on the junta method for extremal combinatorics
initiated by Dinur and Friedgut and further developed by Ellis, Keller,
and the author. 
\end{abstract}

\maketitle

\section{Introduction}

Throughout the paper, we denote $\left[n\right]=\left\{ 1,\ldots,n\right\} ,$
we write $\binom{\left[n\right]}{k}$ for the family of all $k$-element
subsets of $\left[n\right]$, and given a set $S$, we write $\mathcal{P}\left(S\right)$
for the power set of $S$. A family $\f\subseteq\p\left(\left[n\right]\right)$
is called \emph{intersecting} if the intersection of any two sets
in $\f$ is nonempty. A \emph{star} is the family of all sets that
contain a given element.

\emph{Intersection problems for finite sets} study the problem: `how
large can a family of subsets of $\left[n\right]$ be given some restrictions
on the unions and intersections of its elements?' The earliest result
of this class is the Erd\H{o}s-Ko-Rado (EKR) Theorem \cite{erdos1961intersection}
from 1961. 
\begin{thm}[\cite{erdos1961intersection}]
 Let $k\le\frac{n}{2},$ and let $\f\subseteq\binom{\left[n\right]}{k}$
be an intersecting family. Then $\left|\f\right|\le\binom{n-1}{k-1}.$
If $k<\frac{n}{2},$ then $\left|\f\right|=\binom{n-1}{k-1}$ if and
only if $\f$ is a star. 
\end{thm}

Intersection problems for finite sets have become a prolific research
field in extremal combinatorics, and numerous generalizations of the
EKR theorem were obtained. (See the excellent survey of Frankl and
Tokushige \cite{frankl2016survey}). Let us mention one of these generalizations
that we will use in the sequel:

A family $\f\subseteq\binom{\left[n\right]}{k}$ is said to be $s$-wise
intersecting if it does not contain $s$ sets whose intersection is
empty. The following theorem, proved by Frankl \cite{frankl1976sperner},
generalizes the EKR theorem to $s$-wise intersecting families.
\begin{thm}[Frankl, 1976]
\label{thm:Frankl's Theorem} Let $k\le\frac{\left(s-1\right)n}{s},$
and let $\f\subseteq\binom{\left[n\right]}{k}$ be an $s$-wise intersecting
family. Then $\left|\f\right|\le\binom{n-1}{k-1}.$ If $k<\frac{\left(s-1\right)n}{s},$
then $\left|\f\right|=\binom{n-1}{k-1}$ if and only $\f$ is a star.
\end{thm}

\subsection{Set families without a cluster}

We shall be concerned with a generalization of the EKR Theorem, where
the forbidden configuration is known as a \emph{$\left(d,k,s\right)$-cluster}. 
\begin{defn}
A family $\left\{ A_{0},\ldots,A_{d}\right\} \subseteq\binom{\left[n\right]}{k}$
is called a $\left(d,k,s\right)$-cluster if 
\[
\left|A_{0}\cup\cdots\cup A_{d}\right|\le s\text{ and }A_{0}\cap\cdots\cap A_{d}=\varnothing.
\]
 
\end{defn}

We write $f\left(d,k,s,n\right)$ for the largest size of a family
$\mathcal{F}\subseteq\binom{\left[n\right]}{k}$ that does not contain
a $\left(d,k,s\right)$-cluster. Note that $f\left(d,k,s,n\right)\ge\binom{n-1}{k-1}$,
since the star does not contain a $\left(d,k,s\right)$-cluster. 

In this paper we study the following problem.
\begin{problem}
\label{Problem: Cluster} For what values of $d,k,s,n$ do we have
$f\left(d,k,s,n\right)=\binom{n-1}{k-1}$?

This problem generalizes several questions that were studied extensively,
and are still open. Before we discuss the history of the problem,
we give a few basic observations. 
\end{problem}

\begin{enumerate}
\item Problem \ref{Problem: Cluster} makes sense only for $s\ge\frac{d+1}{d}k,$
since no $\left(d,k,s\right)$-cluster exists in $\binom{\left[n\right]}{k}$
if $s<\frac{d+1}{d}k.$ 
\item A $\left(d+1\right)$-wise intersecting family in $\binom{\left[n\right]}{k}$
is free of a $\left(d,k,s\right)$-cluster for any value of $s$.
On the other hand, if $s\ge\min\left(\left(d+1\right)k,n\right),$
then any $d+1$ sets $A_{0},\ldots,A_{d}\in\binom{\left[n\right]}{k}$
whose intersection is empty form a $\left(d,k,s\right)$-cluster.
Hence, for such values of $d,k,s,n$, a family that does not contain
a $\left(d,k,s\right)$-cluster is the same as a $\left(d+1\right)$-wise
intersecting family. Thus, in this case the problem is settled by
Theorem \ref{thm:Frankl's Theorem}.
\item The function $f$ is decreasing in $s$. Combining this fact with
the inequality $f\left(d,k,s,n\right)\ge\binom{n-1}{k-1},$ we obtain
that for any $s_{1}<s_{2}$ such that $f\left(d,k,s_{1},n\right)=\binom{n-1}{k-1}$
we have $f\left(d,k,s_{2},n\right)=\binom{n-1}{k-1}.$
\end{enumerate}
As mentioned above, different special cases of Problem \ref{Problem: Cluster}
were studied in numerous works. In 1980, Katona considered the problem
of determining $f\left(d,k,s,n\right)$ in the case where $d=2$:
\begin{problem}[Katona, 1980]
 How large can a family $\f\subseteq\binom{\left[n\right]}{k}$ be
if $\f$ does not contain sets $A_{1},A_{2},A_{3}$, such that $A_{1}\cap A_{2}\cap A_{3}=\varnothing$
and $\left|A_{1}\cup A_{2}\cup A_{3}\right|\le s$?
\end{problem}

In 1983, Frankl and Füredi \cite{frankl1983new} gave the following
example that shows that the answer to our Problem \ref{Problem: Cluster}
is negative if $s<2k$ and $k\le c\log n$ for a sufficiently small
constant $c$.
\begin{example}
\label{Example Frankl-Furedi} Partition $\left[n\right]$ into sets
$X_{1},\ldots,X_{k}$ of equal size. Let $\mathcal{\mathcal{G}}\subseteq\binom{\left[n\right]}{k}$
be the family of all sets that intersect each $X_{i}$ in a single
vertex. Then $\mathcal{\mathcal{G}}$ is easily seen to be free of
any $\left(d,k,s\right)$-cluster for any $s<2k$. Note that $\left|\mathcal{G}\right|=\left(\frac{n}{k}\right)^{k}$,
and that $\left(\frac{n}{k}\right)^{k}\ge\binom{n-1}{k-1}$, provided
that $k\le c\log n$ for a sufficiently small constant $c$. Hence,
for such $k$ and any $s<2k$ we have $f\left(d,k,s,n\right)>\binom{n-1}{k-1}.$
\end{example}

Frankl and Füredi also showed that $f\left(2,k,2k,n\right)=\binom{n-1}{k-1}$
for any $n\ge k^{2}+3k$ and conjectured that $f\left(2,k,2k,n\right)=\binom{n-1}{k-1}$
for all $k\le\frac{2n}{3}.$ In 2006 Mubayi \cite{mubayi2006erdHos}
proved this conjecture. He also made the following more general conjecture. 
\begin{conjecture}[Mubayi 2006]
\label{Conj Mubayi }Let $d+1\le k\le\frac{d}{d+1}n.$ Then $f\left(d,k,2k,n\right)=\binom{n-1}{k-1}.$ 
\end{conjecture}

In 2007, Mubayi \cite{mubayi2007intersection} proved his conjecture
in the case where $d=3,$ $k$ is fixed and $n$ is sufficiently large.
He has also showed a \emph{stability result} for general fixed $d$.
Specifically, he proved that if $k,d$ are fixed and $n$ tends to
infinity, then any family $\mathcal{F}\subseteq\binom{\left[n\right]}{k}$
that is free of $\left(d,k,2k\right)$-cluster and whose size is $\binom{n-1}{k-1}\left(1-o\left(1\right)\right)$,
must satisfy $\left|\mathcal{F}\backslash\mathcal{S}\right|=o\left(\binom{n-1}{k-1}\right)$
for some star $\mathcal{S}.$ In 2009, Mubayi and Ramadurai \cite{mubayi2009set}
applied Mubayi's stability result and proved that Conjecture \ref{Conj Mubayi }
holds for any fixed $k$ and $d$, provided that $n$ is sufficiently
large. In 2009, Füredi and Özkahya \cite{furedi2009unavoidable} gave
a different proof of the result of Mubayi and Ramadurai and showed
that if $k$ and $d$ are fixed and $n$ is sufficiently large, then
any $\mathcal{F}\subseteq\binom{\left[n\right]}{k}$ whose size is
greater than $\binom{n-1}{k-1}$ contains a special kind of a $\left(d,k,2k\right)$-cluster.
Finally, Keevash and Mubayi \cite{keevash2010set} showed that for
a fixed $d$ and an arbitrarily small $\zeta$, there exists some
$T=T\left(d,\zeta\right)$, such that Conjecture \ref{Conj Mubayi }
holds for any $\zeta n\le k\le\frac{n}{2}-T.$

While Example \ref{Example Frankl-Furedi} shows that we cannot hope
to have $f\left(d,k,s,n\right)=\binom{n-1}{k-1}$ if $k\le c\log n$,
this seems to be a little bit too pessimistic if $k\ge C\log n$ for
a sufficiently large constant $C$. Indeed, for such values of $k$
the family given in Example \ref{Example Frankl-Furedi} is smaller
than the star, and so the equality $f\left(d,k,s,n\right)=\binom{n-1}{k-1}$
becomes possible also for $\frac{d+1}{d}k\le s\le2k$. (See Conjecture
\ref{conj:exact result} and Example \ref{exa: Keevash counter }
below). We show that for $k$ linear in $n$, the equality $f\left(d,k,s,n\right)=\binom{n-1}{k-1}$
holds for any $s\ge\left(\frac{d+1}{d}+\zeta\right)k$ for an arbitrarily
small constant $\zeta>0$, provided that $n$ is sufficiently large.
\begin{thm}
\label{thm:exact cluster result}For any $d\in\mathbb{N},\zeta>0$
there exists $n_{0}=n_{0}\left(d,\zeta\right)$, such that the following
holds. Let $n>n_{0}$, let $\zeta n\le k\le\left(\frac{d}{d+1}-\zeta\right)n,$
and let $\f\subseteq\binom{\left[n\right]}{k}$ be a family that does
not contain a $\left(d,k,\left(\frac{d+1}{d}+\zeta\right)k\right)$-cluster.\textup{\emph{
Then}}\textup{ $\left|\mathcal{F}\right|\le\binom{n-1}{k-1},$ }\textup{\emph{with
equality if and only if $\f$ is a star.}} In particular, Conjecture
\ref{Conj Mubayi } holds for any $\zeta n\le k\le\frac{d+1}{d}n,$
provided that $n\ge n_{0}.$
\end{thm}

This means that for $k$ linear in $n$, not only any family that
is larger than the star must contain a $(d,k,2k)$-cluster as conjectured
by Mubayi, but actually it must contain a $\left(d,k,\left(\frac{d+1}{d}+\zeta\right)k\right)$-cluster,
which is almost the `strongest' cluster we can obtain, due to the
first observation above.

Let $\f,\mathcal{G}\subseteq\binom{\left[n\right]}{k}.$ We say that
$\mathcal{F}$ is \emph{$\epsilon$-essentially contained} in $\mathcal{G}$
if $\left|\mathcal{F}\backslash\mathcal{G}\right|\le\epsilon\binom{n}{k}$
(in words, if a random set in $\binom{\left[n\right]}{k}$ lies in
$\f$ and not in $\mathcal{G}$ with probability at most $\epsilon$).

We also prove a \emph{stability} result for Theorem \ref{thm:exact cluster result}
above.
\begin{thm}
\label{thm:stability cluster result} For any $d\in\mathbb{N},$ and
an arbitrarily small $\zeta>0$, there exists $C>0$, such that the
following holds. Let $\zeta n\le k\le\left(\frac{d}{d+1}-\zeta\right)n,$
and let $\epsilon\ge0$. Suppose that $\mathcal{F}\subseteq\binom{\left[n\right]}{k}$
does not contain a $\left(d,k,\left(\frac{d+1}{d}+\zeta\right)k\right)$-cluster,
and that $\left|\mathcal{F}\right|\ge\binom{n-1}{k-1}\left(1-\epsilon\right)$.
Then $\mathcal{F}$ is $C\epsilon^{1+\frac{1}{C}}$-essentially contained
in a star.
\end{thm}

\subsection{Set families without a simplex-cluster and the Erd\H{o}s-Chvátal
conjecture}

We use Theorem \ref{thm:exact cluster result} to partially resolve
a conjecture of Keevash and Mubayi \cite{keevash2010set} on set families
without a simplex-cluster, and to obtain a new proof for the Erd\H{o}s-Chvátal
simplex conjecture for $k$ linear in $n$.

A \emph{$d$-simplex is a family of $d+1$ sets, such that the intersection
of all of them is empty and the intersection of any $d$ of them is
nonempty.} A \emph{$d$-simplex-cluster} is a $\left(d,k,2k\right)$-cluster
which is also a $d$-simplex. The Erd\H{o}s-Chvátal simplex conjecture
\cite{chvatal1974extremal} states the following. 
\begin{conjecture}[Erd\H{o}s and Chvátal, 1974]
\label{conj:Chvatal}Let $d<k\le\frac{d}{d+1}n.$ Then any family
$\mathcal{F}\subseteq\binom{\left[n\right]}{k}$ that does not contain
a $d$-simplex satisfies $\left|\mathcal{F}\right|\le\binom{n-1}{k-1}$.
Moreover, equality holds if and only if $\mathcal{F}$ is a star. 
\end{conjecture}

In 1976, Frankl \cite{frankl1976sperner} showed that the conjecture
holds if $k\ge\frac{d-1}{d}n$. In 1987, Frankl and Füredi \cite{frankl1987exact}
proved the conjecture in the case where $d$ and $k$ are fixed and
$n\ge n_{0}\left(k,d\right).$ In 2005, Mubayi and Verstraëte \cite{mubayi2005proof}
settled the case $d=2$ for any values of $k$ and $n$. In 2010,
Keevash and Mubayi \cite{keevash2010set} settled the case $\zeta n\le k\le\frac{n}{2}-T,$
for any $\zeta>0$, provided that $T=T\left(\zeta,d\right)$ is sufficiently
large. Finally, Keller and the author \cite{keller2017junta} gave
a 70 pages long proof that Conjecture \ref{conj:Chvatal} holds for
any $k$ in the range (i.e., $d<k\le\frac{d}{d+1}n$), provided that
$n\ge n_{0}\left(d\right)$.

Keevash and Mubayi \cite{keevash2010set} gave the following conjecture
that strengthens both Chvátal's conjecture and Conjecture \ref{Conj Mubayi }.
\begin{conjecture}
\label{conj:Keevash-Mubayi}Let $d<k\le\frac{d}{d+1}n.$ Then any
family $\mathcal{F}\subseteq\binom{\left[n\right]}{k}$ that does
not contain a $d$-simplex-cluster satisfies $\left|\mathcal{F}\right|\le\binom{n-1}{k-1}$.
Moreover, equality holds only if $\mathcal{F}$ is a star. 
\end{conjecture}

Note that if $\left\{ A_{0},\ldots,A_{d}\right\} $ is a $\left(d,k,s\right)$-cluster,
then the intersection of any $d$ of these sets is of size $\ge k-\left(d-1\right)\left(s-k\right).$
Indeed, we have
\begin{align*}
\left|A_{0}\cap\cdots\cap A_{d-1}\right| & \ge\left|A_{0}\right|-\sum_{i=1}^{d-1}\left|A_{0}\backslash A_{i}\right|\ge k-\sum_{i=1}^{d-1}\left|A_{0}\cup\cdots\cup A_{d}\backslash A_{i}\right|\\
 & \ge k-\left(d-1\right)\left(s-k\right).
\end{align*}
 Hence, if $s<\frac{d}{d-1}k$ then the intersection of each $d$
sets in a $\left(d,k,s\right)$-cluster is of size larger than $k-\left(d-1\right)\left(\frac{d}{d-1}k-k\right)=0.$
Thus, for such $s$, any $\left(d,k,s\right)$-cluster is a $d$-simplex.
Therefore, as an immediate corollary of Theorem \ref{thm:exact cluster result}
we obtain that Conjecture \ref{conj:Keevash-Mubayi} holds for all
$k\ge\zeta n$, provided that $n\ge n_{0}\left(\zeta,d\right)$. 
\begin{thm}
\label{cor:Simplex-Cluster} For each $d\in\mathbb{N},\zeta>0,$ there
exists $n_{0}=n_{0}\left(d,\zeta\right)$, such that the following
holds. Let $n\ge n_{0}\left(\zeta,d\right),$ and let $\zeta n<k\le\frac{d}{d+1}n.$
Then any family $\mathcal{F}\subseteq\binom{\left[n\right]}{k}$ that
does not contain a $d$-simplex-cluster, satisfies $\left|\mathcal{F}\right|\le\binom{n-1}{k-1}$.
Moreover, equality holds if and only if $\mathcal{F}$ is a star. 
\end{thm}

Since Conjecture \ref{conj:Keevash-Mubayi} strengthens the Erd\H{o}s-Chvátal
conjecture \ref{conj:Chvatal}, this paper gives a relatively short
proof of the fact that the Erd\H{o}s-Chvátal conjecture holds for
all $k\ge\zeta n$, provided that $n\ge n_{0}\left(\zeta,d\right).$

\subsection{The proof method }

The main tool we use in our proof is the `junta method', initiated
by Dinur and Friedgut \cite{dinur2009intersecting} and further developed
by Keller and the author \cite{keller2017junta}, and by Ellis, Keller,
and the author \cite{ellis2016stabilityfor}. One of our goals in
writing this paper is to make this recent technique more accessible,
by providing a shorter paper that follows the framework of the junta
method. 

\section{Juntas and proof sketch}

Let $j<k<n$. A family $\j\subseteq\binom{\left[n\right]}{k}$ is
said to be a \emph{$j$-junta} if there exists a set $J$ of size
$j$, and a family $\g\subseteq\p\left(J\right)$, such that a set
$A$ is in $\j$ if and only if $A\cap J$ is in $\g.$ Informally,
a family is a \emph{junta} if it is  a $j$-junta for a constant $j$
independent of $k$ and $n$.

The notion `junta' originates in the field known as `analysis of Boolean
functions', where it plays a central role (see e.g., Bourgain \cite{bourgain2002distribution},
Dinur et al. \cite{dinur2006fourier}, Friedgut \cite{friedgut1998boolean},
and Kindler-Safra \cite{kindler2002noise}). 

They were introduced to extremal combinatorics by Dinur and Friedgut
\cite{dinur2009intersecting}. They showed that any intersecting family
is essentially contained in an intersecting junta. 
\begin{thm}[Dinur and Friedgut 2009]
\label{thm: Dinur and Friedgut} For any $r$ there exists a $C=C\left(r\right),j=j\left(r\right)$,
such that any intersecting family $\f\subseteq\binom{\left[n\right]}{k}$
is $C\left(\frac{k}{n}\right)^{r}$-essentially contained in an intersecting
$j$-junta.
\end{thm}

Note that Theorem \ref{thm: Dinur and Friedgut} is trivial if $\frac{k}{n}=\Theta\left(1\right)$.
In that regime, they managed to show a slightly weaker version of
the following recent result of Friedgut and Regev \cite{friedgut2017kneser}.
\begin{thm}[\cite{friedgut2017kneser}]
 \label{thm:Friedgut and Regev}For each $\zeta,\epsilon>0$ there
exists $j=j\left(\zeta,\epsilon\right)\in\mathbb{N}$, such that the
following holds. Let $\zeta n<k<\left(\frac{1}{2}-\zeta\right)n$
and let $\f\subseteq\binom{\left[n\right]}{k}$ be an intersecting
family. Then $\f$ is $\epsilon$-essentially contained in an intersecting
$j$-junta.
\end{thm}

These results inspired the works of Ellis, Keller, and the author
\cite{ellis2016stabilityfor,keller2017junta} who developed a method
to show that the extremal family that is free of a certain forbidden
configuration is some junta $\mathcal{J}_{\mathrm{ex}}.$ The method
is combined of the following ingredients. 
\begin{itemize}
\item \textbf{Ingredient 1: }Show that any family is essentially contained
in a junta that does not contain the forbidden configuration. 
\item \textbf{Ingredient 2: }Show that $\j_{\mathrm{ex}}$ is the largest
junta that does not contain the forbidden configuration, and prove
a stability result of this statement. I.e. if $\j$ is a junta that
does not contain the forbidden configuration and whose size is close
to $\left|\j_{\mathrm{ex}}\right|$, then $\j$ is essentially contained
in $\j_{\mathrm{ex}}.$ 
\item \textbf{Ingredient 3: }Show that if $\f$ is a small alteration of
$\j_{\mathrm{ex}}$ that does not contain the forbidden configuration,
then $\left|\f\right|\le\left|\j_{\mathrm{ex}}\right|.$ 
\end{itemize}
These ingredients fit together to show that $\j_{\mathrm{ex}}$ is
the extremal family that does not contain the forbidden configuration.
Indeed, suppose that $\f$ is the extremal family. Then the first
ingredient yields that $\f$ is essentially contained in a junta $\j.$
In particular, the size of $\j$ is not much smaller than the size
of $\f,$ which is greater or equal to the size of $\j_{\mathrm{ex}}.$
The second ingredient implies that $\j$ is essentially contained
in $\j_{\mathrm{ex}}$, and hence $\f$ is essentially contained in
$\j_{\mathrm{ex}}.$ The third ingredient implies that $\left|\f\right|\le\left|\j_{\mathrm{ex}}\right|$,
and therefore $\j_{\mathrm{ex}}$ is the extremal family. 

In our case, the forbidden configuration is a $\left(d,k,\left(\frac{d+1}{d}+\zeta\right)k\right)$-cluster,
and the junta $\mathcal{J}_{\mathrm{ex}}$ is a star.

\textbf{Showing that the largest junta free of a $\left(d,k,\left(\frac{d+1}{d}+\zeta\right)k\right)$-cluster
is a star, and proving stability.}

We observe that any $j$-junta that does not contain a $\left(d,k,\left(\frac{d+1}{d}+\zeta\right)k\right)$-cluster
is actually $\left(d+1\right)$-wise intersecting, provided that $k\ge k_{0}\left(j\right)$.
Then, Ingredient 2 amounts to proving a stability result for Frankl's
Theorem (Theorem \ref{thm:Frankl's Theorem}), i.e. to showing that
a $\left(d+1\right)$-wise intersecting family whose size is close
to $\binom{n-1}{k-1}$ is close to a star. This was proved by Ellis,
Keller, and the author \cite{ellis2016stability}.

\textbf{Showing that any family free of a $\left(d,k,\left(\frac{d+1}{d}+\zeta\right)k\right)$-cluster
is essentially contained in a $\left(d+1\right)$-wise intersecting
junta. }

The proof is based on the regularity method and it goes as follows. 
\begin{enumerate}
\item Note that each set $J$ of constant size decomposes the sets in $\mathcal{F}$
into $2^{\left|J\right|}$ parts according to their intersection with
$J$. The first step is to show that we may find a set $J$ of constant
size, such that $\f$ is a union of parts that satisfy a certain quasirandomness
notion and a sufficiently small remainder that can be ignored. 
\item We then take our approximating junta to consist of the union of the
parts that satisfy the quasirandomness notion. The small size of the
remainder translates into the fact that $\f$ is essentially contained
in $\j$, and our goal becomes to show that $\j$ is $\left(d+1\right)$-wise
intersecting. The second step is to turn this task into a statement
about the quasirandom parts. Namely, we obtain that it is enough to
show that if $\f_{0},\ldots,\f_{d}$ are quasirandom families, then
they mutually contain a $\left(d,k,\left(\frac{d+1}{d}+\zeta\right)k\right)$-cluster,
provided that $n\ge n_{0}\left(\zeta\right)$, i.e. it is enough to
show that there exists sets $A_{0}\in\f_{0},\ldots,A_{d}\in\f_{d}$
whose intersection is empty, such that $\left|A_{0}\cup\cdots\cup A_{d}\right|\le\left(\frac{d+1}{d}+\zeta\right)k$.
The next steps concern this new task.
\item We choose an $l=k\left(1+\zeta'\right)$ for a small constant $\zeta'>0$,
and we write $\mathcal{F}_{i}^{\uparrow l}$ for the family of all
sets in $\binom{\left[n\right]}{l}$ that contain a set in $\f_{i}.$
The third step is to show that the probability that a random set in
$\binom{\left[n\right]}{l}$ is in $\mathcal{F}_{i}^{\uparrow l}$
is close to 1.
\item The fourth step is to give a simple union bound that shows that the
families $\f_{0}^{\uparrow l},\ldots,\f_{d+1}^{\uparrow l}$ mutually
contain a random $\left(d,l,\frac{d+1}{d}l\right)$-cluster with positive
probability. 
\item The last step is to deduce from the $\left(d,l,\frac{d+1}{d}l\right)$-cluster
appearing in the families $\f_{0}^{\uparrow l},\ldots,\f_{d+1}^{\uparrow l}$,
that a $\left(d,k,\left(\frac{d+1}{d}+\zeta\right)k\right)$-cluster
appears in the quasirandom families $\mathcal{F}_{0},\ldots,\f_{l}.$ 
\end{enumerate}
\textbf{Showing that the star is the largest family free of a }$\left(d,k,\left(\frac{d+1}{d}+\zeta\right)k\right)$\textbf{-cluster
in its neighborhood}.

Finally, we shall give an argument based on the Kruskal-Katona Theorem
\cite{katona1964intersection,kruskal1963number} to accomplish the
third ingredient, i.e. we show that if a family that does not contain
a $\left(d,k,\left(\frac{d+1}{d}+\zeta\right)k\right)$-cluster is
close to a star, then its size must be smaller than it. Let $\mathcal{F}\subseteq\binom{\left[n\right]}{k}$
be a family close to a star $\mathcal{S}$. We start by decomposing
$\mathcal{F}$ into the large family $\f_{1}:=\f\cap\mathcal{S}$
inside the star and the small family $\f_{0}:=\f\backslash\mathcal{S}$
outside of it. One can think of $\mathcal{F}$ as a family constructed
from $\mathcal{S}$ by first adding the element of $\f_{0}$ into
$\f$, thereby unavoidably putting $\left(d,k,\left(\frac{d+1}{d}+\zeta\right)k\right)$-clusters
inside it, and then removing the elements of $\mathcal{S}\backslash\f_{1}$
out of $\f$ to destroy all of these copies. Our goal then becomes
to show that $\f_{0}$ is negligible compared to $\left|\mathcal{S}\backslash\f_{1}\right|.$
The proof follows the following steps:
\begin{enumerate}
\item We choose $l$ slightly larger than $k$, and we use the Kruskal-Katona
Theorem (Theorem \ref{thm:Kruskal-Katona} bellow) to give a lower
bound on $\left|\mathcal{F}_{1}^{\uparrow l}\right|$ in terms of
$\left|\f_{1}\right|.$ 
\item We observe that the families $\f_{1}^{\uparrow l},\ldots,\f_{1}^{\uparrow l},\f_{0}^{\uparrow l}$
do not mutually contain a $\left(d,l,\frac{d+1}{d}l\right)$-cluster,
and we use this fact to deduce an upper bound on $\left|\f_{0}^{\uparrow l}\right|$
in terms of $\left|\f_{1}^{\uparrow l}\right|.$
\item We apply the Kruskal-Katona Theorem again in order to upper bound
the size of $\mathcal{F}_{0}$ in term of $\left|\f_{0}^{\uparrow l}\right|.$ 

Combining all these upper bound we obtain an upper bound of $\left|\f_{0}\right|$
in terms of $\left|\f_{1}\right|.$ It turns out that this upper bound
is sufficient to complete the proof. 
\end{enumerate}

\subsection{Notations }

We use bold letters to denote random variables. Let $X$ be some set.
We write $\boldsymbol{A}\sim\binom{X}{k}$ to denote that $\boldsymbol{A}$
is a uniformly random $k$-set in $X$. Let $\mathcal{F}\subseteq\binom{X}{k}$
be some family. We write 
\[
\mu\left(\f\right)=\frac{\left|\f\right|}{\binom{\left|X\right|}{k}}=\Pr_{\boldsymbol{A}\sim\binom{X}{k}}\left[\boldsymbol{A}\in\f\right].
\]
Given a set $J\subseteq X,$ and $B\subseteq J$, we write $\f_{J}^{B}$
for the family $\left\{ A\in\binom{X\backslash J}{k-\left|B\right|}|\,A\cup B\in\f\right\} .$
We therefore have 
\[
\mu\left(\f_{J}^{B}\right)=\Pr_{\boldsymbol{A}\sim\binom{\left[n\right]}{k}}\left[\boldsymbol{A}\in\f|\boldsymbol{A}\cap J=B\right].
\]
 Let $J\subseteq\left[n\right]$, and let $\g\subseteq\p\left(J\right)$
be some family. We write $\left\langle \g\right\rangle $ for the
$\left|J\right|$-junta of all the sets $A\in\binom{\left[n\right]}{k}$
such that $A\cap J$ is in $\g.$ We call $\left\langle \g\right\rangle $
the \emph{junta generated} by $\g.$ 

A family $\mathcal{A}$ is said to be monotone if for any $A\in\mathcal{A}$
and any $B\supseteq A$ we have $B\in\mathcal{A}.$ The \emph{monotone
closure of $\f$, denoted by} $\f^{\uparrow}$, is the monotone family
of all sets in $\p\left(n\right)$ that contain a set in $\f.$ Hence,
$\f^{\uparrow l}=\f^{\uparrow}\cap\binom{\left[n\right]}{l}$. 

The $p$-biased measure is a probability distribution on sets $\boldsymbol{A}\sim\p\left(\left[n\right]\right),$
where we put each element $i$ in $\boldsymbol{A}$ independently
with probability $p.$ For a family $\mathcal{A}\subseteq\p\left(\left[n\right]\right)$
we write 
\[
\mu_{p}\left(\mathcal{A}\right)=\Pr_{\boldsymbol{A}\sim\mu_{p}}\left[\boldsymbol{A}\in\mathcal{A}\right].
\]

\section{Consequences of the Kruskal-Katona Theorem }

The Kruskal-Katona Theorem gives us a lower bound on $\left|\f^{\uparrow l}\right|$
in terms of $\left|\mathcal{\f}\right|.$ Before stating it we shall
give a trivial lower bound that would also be useful to us.
\begin{lem}
\label{lem:trivial kk}Let $k<l<n$ be some natural numbers and let
$\mathcal{\f}\subseteq\binom{\left[n\right]}{k}$ be some family.
Then $\mu\left(\mathcal{\f}\right)\le\mu\left(\mathcal{\f}^{\uparrow l}\right).$ 
\end{lem}

\begin{proof}
Choose a set $\boldsymbol{A}\sim\binom{\left[n\right]}{k}$ and choose
a set $\boldsymbol{B}\sim\binom{\left[n\right]\backslash\boldsymbol{A}}{l-k}$.
We have 
\[
\mu\left(\mathcal{F}^{\uparrow l}\right)=\Pr\left[\boldsymbol{A}\cup\boldsymbol{B}\in\mathcal{F}^{\uparrow l}\right]\ge\Pr\left[\boldsymbol{A}\in\mathcal{F}\right]=\mu\left(\mathcal{F}\right).
\]
 
\end{proof}
The lexicographically ordering on $\binom{\left[n\right]}{k}$ is
the ordering on sets defined by $A<_{L}B$ if $\min\left\{ A\Delta B\right\} \in A.$
We let $\mathcal{L}\left(i,k,n\right)$ be the family of the $i$
sets in $\binom{\left[n\right]}{k}$ that are first in the lexicographic
ordering. Thus, $\mathcal{L}\left(\binom{n-1}{k-1},k,n\right)$ is
the star of all sets that contain the element 1. The Kruskal-Katona
\cite{katona2009theorem,kruskal1963number} Theorem is known to be
equivalent to the following:
\begin{thm}[Kruskal-Katona]
\label{thm:Kruskal-Katona} Let $k<l<n,$ and let $i\le\binom{\left[n\right]}{k}$.
Suppose that $\left|\mathcal{F}\right|\ge\left|\mathcal{L}^{\left(i,k,n\right)}\right|.$
Then $\left|\mathcal{F}^{\uparrow l}\right|\ge\left|\left(\mathcal{L}^{\left(i,k,n\right)}\right)^{\uparrow l}\right|.$ 
\end{thm}

We shall make use of the following corollaries of the Kruskal-Katona
theorem. 
\begin{cor}
\label{cor:kk up}Let $k<l<n$, let $\epsilon>0,$ let $\mathcal{F}\subseteq\binom{\left[n\right]}{k}$
be some family, and let $\mathcal{G}=\mathcal{F}^{\uparrow l}$. 
\begin{enumerate}
\item If $\left|\mathcal{F}\right|\ge\binom{n-1}{k-1}\left(1-\epsilon\right),$
then $\left|\mathcal{G}\right|\ge\binom{n-1}{l-1}\left(1-\epsilon\right)$. 
\item If we moreover have $\mathcal{\left|\mathcal{F}\right|}\ge\binom{n-1}{k-1}-\left(1-\epsilon\right)\binom{n-m}{k-1},$
then $\left|\mathcal{G}\right|\ge\binom{n-1}{l-1}-\left(1-\epsilon\right)\binom{n-m}{l-1}.$ 
\item Consequently, for each $\zeta>0$ there exists a constant $C>0$,
such that the following holds. Suppose that $k,l\in\left(\zeta n,\left(1-\zeta\right)n\right),$
and that $l-k>\zeta n$. If $\left|\f\right|\ge\binom{n-1}{k-1}\left(1-\epsilon\right),$
then $\left|\g\right|\ge\binom{n-1}{l-1}\left(1-C\epsilon^{1+\frac{1}{C}}\right).$
\end{enumerate}
\end{cor}

\begin{proof}
The Kruskal-Katona Theorem implies that it suffices to prove the corollary
in the case where $\mathcal{F}=\mathcal{L}\left(i,k,n\right).$ 

\textbf{Proving (1).} Note that we may suppose that $\left|\f\right|\le\binom{n-1}{k-1}.$
Let $\mathcal{S}$ be the star of all elements containing 1. Since
any set in $\mathcal{S}$ is (lexicographically) smaller than any
set not in $\mathcal{S}$, the family $\mathcal{F}$ is contained
in $\mathcal{S}$. 

Thus, $\mu\left(\mathcal{F}_{\left\{ 1\right\} }^{\left\{ 1\right\} }\right)\ge1-\epsilon,$
and so Lemma \ref{lem:trivial kk} implies that $\mu\left(\left(\mathcal{F}_{\left\{ 1\right\} }^{\left\{ 1\right\} }\right)^{\uparrow\left(l-1\right)}\right)\ge1-\epsilon.$
Note that we have $A\in\left(\mathcal{F}_{\left\{ 1\right\} }^{\left\{ 1\right\} }\right)^{\uparrow l-1}$
if and only if $A\cup\left\{ 1\right\} \in\mathcal{G}.$ Hence, 
\[
\left|\mathcal{G}\right|=\left|\left(\mathcal{F}_{\left\{ 1\right\} }^{\left\{ 1\right\} }\right)^{\uparrow\left(l-1\right)}\right|=\binom{n-1}{l-1}\mu\left(\left(\mathcal{F}_{\left\{ 1\right\} }^{\left\{ 1\right\} }\right)^{\uparrow\left(l-1\right)}\right)\ge\binom{n-1}{l-1}\left(1-\epsilon\right),
\]
 as desired. 

\textbf{Proving (2).} Again we may assume that $\left|\f\right|\le\binom{n-1}{k-1}$.
Write $i=\binom{n-1}{k-1}-\binom{n-m}{k-1}$, and note that $\mathcal{L}\left(i,k,n\right)$
is the family 
\[
\left\{ A\in\binom{\left[n\right]}{k}|\,1\in A\,,\,A\cap\left[2,\ldots,m\right]\ne\varnothing\right\} .
\]
 Since $\left|\f\right|>i$ we obtain that $\f\supseteq\mathcal{L}\left(i,k,n\right).$
Additionally, the intersection of any sets in $\f\backslash\mathcal{L}\left(i,k,n\right)$
with the set $\left[m\right]$ is the set $\left\{ 1\right\} .$ Therefore,
\[
\mu\left(\mathcal{\f}_{\left[m\right]}^{\left\{ 1\right\} }\right)=\frac{\left|\f\right|-i}{\binom{n-m}{k-1}}\ge\epsilon.
\]
 By Lemma \ref{lem:trivial kk}, $\mu\left(\left(\f_{\left[m\right]}^{\left\{ 1\right\} }\right)^{\uparrow l-1}\right)\ge\epsilon$. 

Write $j=\binom{n-1}{l-1}-\binom{n-m}{l-1}.$ Note that, similarly
to the family $\f$, the family $\mathcal{G}$ contains the family
\[
\mathcal{L}\left(j,l,n\right)=\left\{ A\in\binom{\left[n\right]}{l}|\,1\in A,\,A\cap\left[2,\ldots,m\right]\ne\varnothing\right\} .
\]
 Moreover, all the elements of $\mathcal{G}\backslash\mathcal{L}\left(j,l,n\right)$
are the elements of the form $A\cup\left\{ 1\right\} $, where $A\in\left(\mathcal{\f}_{\left[m\right]}^{\left\{ 1\right\} }\right)^{\uparrow l-1}$.
Therefore,
\[
\epsilon\le\mu\left(\left(\mathcal{\f}_{\left[m\right]}^{\left\{ 1\right\} }\right)^{\uparrow\left(l-1\right)}\right)=\mu\left(\g_{\left[m\right]}^{\left\{ 1\right\} }\right)=\frac{\left|\g\right|-j}{\binom{n-m}{l-1}}.
\]
 Rearranging and substituting the value of $j$, we have 
\[
\left|\mathcal{G}\right|\ge\binom{n-1}{l-1}-\left(1-\epsilon\right)\binom{n-m}{l-1}.
\]
 This completes the proof of (2).

\textbf{Deducing (3) from (2).} Let $m$ be maximal with $\left|\f\right|\ge\binom{n-1}{k-1}-\binom{n-m}{k-1},$
and write 
\[
\mathcal{\left|\mathcal{F}\right|}=\binom{n-1}{k-1}-\left(1-\epsilon'\right)\binom{n-m}{k-1}=\binom{n-1}{k-1}\left(1-\frac{\left(1-\epsilon'\right)\binom{n-m}{k-1}}{\binom{n-1}{k-1}}\right).
\]
 By (2), 
\[
\left|\g\right|\ge\binom{n-1}{l-1}-\left(1-\epsilon'\right)\binom{n-m}{l-1}=\binom{n-1}{l-1}\left(1-\frac{\binom{n-m}{l-1}\left(1-\epsilon'\right)}{\binom{n-1}{l-1}}\right).
\]

Hence, to complete the proof we must show that 
\begin{equation}
\frac{\binom{n-m}{l-1}\left(1-\epsilon'\right)}{\binom{n-1}{l-1}}\le C\left(\frac{\left(1-\epsilon'\right)\binom{n-m}{k-1}}{\binom{n-1}{k-1}}\right)^{1+\frac{1}{C}},\label{eq:desired}
\end{equation}
 provided that $C=C\left(\zeta\right)$ is sufficiently large. 

\textbf{Getting rid of $\epsilon'$. }We shall now show that the $\left(1-\epsilon'\right)$-terms
of (\ref{eq:desired}) get swallowed by the constant $C$, i.e $\left(1-\epsilon'\right)=\Theta_{\zeta}\left(1\right).$
We may assume that $n-m\ge l-1$, for otherwise the left hand side
of (\ref{eq:desired}) is $0$. By the definition of $m$, 
\[
\left(1-\epsilon'\right)\binom{n-m}{k-1}\ge\binom{n-m-1}{k-1}=\left(1-\frac{k-1}{n-m}\right)\binom{n-m}{k-1}.
\]
 Hence, 
\begin{equation}
1-\epsilon'\ge1-\frac{k-1}{n-m}\ge1-\frac{k-1}{l-1}=\frac{l-k}{l-1}\ge\frac{l-k}{n}\ge\zeta.\label{eq:zeta}
\end{equation}
 This completes the proof that $\left(1-\epsilon'\right)=O_{\zeta}\left(1\right)$,
and so it is enough to show that 
\begin{equation}
\frac{\binom{n-m}{l-1}}{\binom{n-1}{l-1}}\le C\left(\frac{\binom{n-m}{k-1}}{\binom{n-1}{k-1}}\right)^{1+\frac{1}{C}},\label{eq:reduced eq}
\end{equation}
 provided that $C$ is sufficiently large. 

\textbf{Showing (\ref{eq:reduced eq}).} Rearranging (\ref{eq:reduced eq}),
our goal becomes to show that 
\[
\frac{\binom{n-m}{l-1}/\binom{n-1}{l-1}}{\binom{n-m}{k-1}/\binom{n-1}{k-1}}\le C\left(\frac{\binom{n-m}{k-1}}{\binom{n-1}{k-1}}\right)^{\frac{1}{C}}.
\]

This would follow once we show that that:
\begin{equation}
\frac{\binom{n-m}{k-1}}{\binom{n-1}{k-1}}=\left(1-\frac{k-1}{n-1}\right)\left(1-\frac{k-1}{n-2}\right)\cdots\left(1-\frac{k-1}{n-m+1}\right)\ge C'^{m-1},\label{eq: technical KK}
\end{equation}
 and
\begin{equation}
\frac{\binom{n-m}{l-1}/\binom{n-1}{l-1}}{\binom{n-m}{k-1}/\binom{n-1}{k-1}}=\frac{\left(1-\frac{l-1}{n-1}\right)\left(1-\frac{l-1}{n-2}\right)\cdots\left(1-\frac{l-1}{n-m+1}\right)}{\left(1-\frac{k-1}{n-1}\right)\left(1-\frac{k-1}{n-2}\right)\cdots\left(1-\frac{k-1}{n-m+1}\right)}\le C''^{m-1}\label{eq:technical kk2}
\end{equation}
 where $0<C',C''<1$ are constants depending only on $\zeta.$ 

Now note there are $m-1$ terms in the middle of (\ref{eq: technical KK})
and each is greater than $1-\frac{k-1}{n-m},$ which is greater than
$\zeta$ by (\ref{eq:zeta}). Similarly, there are $m-1$ terms in
the middle of (\ref{eq:technical kk2}), and each term satisfies
\[
\frac{1-\frac{l-1}{n-i}}{1-\frac{k-1}{n-i}}=1-\frac{\frac{l-k}{n-i}}{1-\frac{k-1}{n-i}}\le1-\frac{l-k}{n}\le1-\zeta.
\]
 This completes the proof of the lemma.
\end{proof}

\section{Proof of the approximation by junta result and of the stability result}

In this section we shall prove a stability result that says that any
family that does not contain a $\left(d,k,\left(\frac{d+1}{d}+\zeta\right)k\right)$-cluster
whose size close to that of a star must in itself be close to a star. 
\begin{prop}
\label{prop:Rough stability result}For each $\zeta,\epsilon>0$ there
exists $\delta>0,n_{0}\in\mathbb{N}$, such that the following holds.
Let $n>n_{0}$, let $\zeta<\frac{k}{n}<\frac{d}{d+1}-\zeta$, and
let $\f\subseteq\binom{\left[n\right]}{k}$ be some family that does
not contain a $\left(d,k,\left(\frac{d+1}{d}+\zeta\right)k\right)$-cluster.
If $\left|\f\right|\ge\binom{n-1}{k-1}\left(1-\delta\right),$ then
$\f$ is $\epsilon$-essentially contained in a star. 
\end{prop}

Note that Proposition \ref{prop:Rough stability result} is a weaker
version of Theorem \ref{thm:stability cluster result}. However, we
shall show that the `weak' Proposition \ref{prop:Rough stability result}
can be bootstrapped into the stronger Theorem \ref{thm:stability cluster result}
in Section \ref{sec:exact result}. 

This section is divided into three parts. 
\begin{enumerate}
\item We first show that a junta is free of a $\left(d,k,\left(\frac{d+1}{d}+\zeta\right)k\right)$-cluster
if and only if it is  $\left(d+1\right)$-wise intersecting. This
part is needed only for motivational purposes and we shall not use
this fact.
\item We then show that any family that is  free of a $\left(d,k,\left(\frac{d+1}{d}+\zeta\right)k\right)$-cluster
is essentially contained in a $\left(d+1\right)$-wise intersecting
junta.
\item Finally, we shall apply a stability result of Theorem \ref{thm:Frankl's Theorem}
by \cite{ellis2016stability} to deduce Proposition \ref{prop:Rough stability result}.
\end{enumerate}

\subsection{Any junta that is  free of a $\left(d,k,\left(\frac{d+1}{d}+\zeta\right)k\right)$-cluster
is $\left(d+1\right)$-wise intersecting }

We now show that a junta does not contain a $\left(d,k,\left(\frac{d+1}{d}+\zeta\right)k\right)$-cluster
if and only if it is $\left(d+1\right)$-wise intersecting.
\begin{prop}
Let $j>0$, let $s\ge\frac{d+1}{d}k+j,$ and let $n\ge s$. Then a
$j$-junta $\j\subseteq\binom{\left[n\right]}{k}$ is free of a $\left(d,k,s\right)$-cluster
if and only if it is $\left(d+1\right)$-wise intersecting. 
\end{prop}

\begin{proof}
Note that any $\left(d+1\right)$-wise intersecting family is free
of a $\left(d,k,s\right)$-cluster. So suppose on the contrary that
$\j$ is a $j$-junta that does not contains a $\left(d,k,s\right)$-cluster
and is not $\left(d+1\right)$-wise intersecting. Let $J$ be some
$j$-set and let $\g\subseteq\p\left(J\right)$ be a family, such
that a set $A$ is in $\j$ if and only if $A\cap J$ is in $\g.$
Let $A_{0},\ldots,A_{d}\in\j$ be some sets whose intersection is
empty, and let $S\subseteq\left[n\right]\backslash J$ be some set
of size $s-j.$ Since $\left|S\right|\ge\frac{d+1}{d}k$, it is easy
to see that there exists sets 
\[
B_{0}\in\binom{S}{k-\left|A_{0}\cap J\right|},B_{1}\in\binom{S}{k-\left|A_{1}\cap J\right|},\ldots,B_{d}\in\binom{S}{k-\left|A_{d}\cap J\right|},
\]
 such that $B_{0}\cap\cdots\cap B_{d}=\varnothing.$ Now the sets
$B_{0}\cup\left(A_{0}\cap J\right),\ldots,B_{d}\cup\left(A_{d}\cap J\right)$
form a $\left(d,k,s\right)$-cluster in $\j,$ contradicting the hypothesis
that $\j$ does not contain a $\left(d,k,s\right)$-cluster.
\end{proof}
Our goal will now be to prove that any family that does not contain
a $\left(d,k,\left(\frac{d+1}{d}+\zeta\right)k\right)$-cluster is
essentially contained in a $\left(d+1\right)$-wise intersecting junta.
Our first ingredient is the following regularity lemma of \cite{ellis2016stabilityfor}.

\subsection{The regularity lemma of \cite{ellis2016stabilityfor}}

A family $\f\subseteq\binom{\left[n\right]}{k}$ is said to be $\left(r,\epsilon\right)$-regular.
If $\left|\mu\left(\f_{J}^{B}\right)-\mu\left(\f\right)\right|\le\epsilon$
for any $J$ of size at most $r$ and any $B\subseteq J.$ 

As mentioned, every set $J$ decomposes $\f$ into the $2^{\left|J\right|}$
parts $\left\{ \f_{J}^{B}\right\} _{B\subseteq J}.$ The following
regularity lemma of \cite{ellis2016stabilityfor} allows us to find
a set $J$ that decomposes our family into some $\left(r,\epsilon\right)$-regular
parts and some `negligible' parts that together contribute very little
to the measure of $\f$. 
\begin{thm}[\cite{ellis2016stabilityfor} Theorem 1.7]
\label{thm:regularity} For each $\delta,\epsilon,\zeta>0$ there
exists $j\in\mathbb{N}$, such that the following holds. Let $\zeta n\le k\le\left(1-\zeta\right)n$
and let $\f\subseteq\binom{\left[n\right]}{k}$ be a family. Then
there exists a set $J$ of size $j$ and a family $\g\subseteq\p\left(J\right)$
such that the following holds. 
\begin{enumerate}
\item For each $B\in\j$, the family $\f_{J}^{B}$ is $\left(\left\lceil \frac{1}{\delta}\right\rceil ,\delta\right)$-regular
and $\mu\left(\f_{J}^{B}\right)>\frac{\epsilon}{2}$.
\item The family $\f$ is $\epsilon$-essentially contained in the $j$-junta
$\left\langle \g\right\rangle .$
\end{enumerate}
\end{thm}

\subsection{If $\protect\f$ is $\left(\left\lceil \frac{1}{\delta}\right\rceil ,\delta\right)$-regular
and $l\ge k\left(1+\zeta\right)$, then $\mu\left(\protect\f^{\uparrow l}\right)$
is close to 1 }

Let $\frac{k}{n}<\frac{l}{n}$ be some numbers that are bounded away
from $0,1$ and each other. In order to prove our approximation by
junta theorem, we will need to show that if $\f\subseteq\binom{\left[n\right]}{k}$
is a $\left(\left\lceil \frac{1}{\delta}\right\rceil ,\delta\right)$-regular
family, then either $\mu\left(\f\right)$ is close to 0 or $\mu\left(\f^{\uparrow l}\right)$
is close to 1. This lemma is in the spirit of Dinur and Friedgut \cite[Lemma 3.2]{dinur2009intersecting}. 
\begin{lem}
\label{lem: Follows from Friedgut-Junta Theorem} For each $\zeta>0,$
there exists $\delta>0$ such that the following holds. Let $\zeta<\frac{l}{n}<1-\zeta$,
let $\frac{k}{n}<\frac{l}{n}-\frac{\zeta}{2}$, and suppose that $\f\subseteq\binom{\left[n\right]}{k}$
is a $\left(\left\lceil \frac{1}{\delta}\right\rceil ,\delta\right)$-regular
family. Then either $\mu\left(\f\right)<\epsilon$ or $\mu\left(\f^{\uparrow l}\right)>1-\epsilon.$ 
\end{lem}

One of the main tools is the following well known corollary of Friedgut's
Junta Theorem \cite{friedgut1998boolean} and Russo's Lemma \cite{russo1982approximate}.
\begin{thm}[Friedgut's junta Theorem for monotone families]
 For each $\epsilon,\zeta,C>0$ there exists $j\in\mathbb{N}$, such
that the following holds. Let $\f\subseteq\p\left(\left[n\right]\right)$
be a monotone family, let $p\in\left(\zeta,1-\zeta\right)$, and suppose
that $\frac{d\mu_{p}\left(\f\right)}{dp}\le C.$ The there exists
a $j$-junta $\mathcal{J}$, such that $\mu_{p}\left(\f\Delta\mathcal{J}\right)<\epsilon.$
\end{thm}

\begin{proof}[Proof of Lemma \ref{lem: Follows from Friedgut-Junta Theorem}]
 Note that we may assume that $n$ is sufficiently large by decreasing
$\delta$ if necessary. By Lemma \ref{lem:trivial kk}, we have $\mu\left(\f^{\uparrow r}\right)\le\mu\left(\f^{\uparrow l}\right)$
for any $r\le l.$ Hence, for any $p\le\frac{k+l}{2n}$ we have 
\begin{align*}
\mu_{p}\left(\f^{\uparrow}\right) & =\sum_{r=0}^{n}p^{r}\left(1-p\right)^{n-r}\binom{n}{r}\mu\left(\f^{\uparrow r}\right)\\
 & \le\sum_{r=0}^{l}p^{r}\left(1-p\right)^{n-r}\binom{n}{r}\mu\left(\f^{\uparrow l}\right)+\Pr_{r\sim\mathrm{Bin}\left(n,p\right)}\left[r\ge l\right]\\
 & \le\mu\left(\f^{\uparrow l}\right)+\Pr_{r\sim\mathrm{Bin}\left(n,p\right)}\left[r\ge l\right].
\end{align*}
 Suppose on the contrary that $\mu\left(\f^{\uparrow l}\right)\ge1-\epsilon.$
A simple Chernoff bound implies that $\Pr_{r\sim\mathrm{Bin}\left(n,p\right)}\left[r\ge l\right]<\frac{\epsilon}{2}$,
provided that $n\ge n_{0}\left(\zeta\right).$ Thus, 
\[
\mu_{p}\left(\f^{\uparrow}\right)\le1-\epsilon+\frac{\epsilon}{2}=1-\frac{\epsilon}{2}
\]
 for any $p\le\frac{k+l}{2n}$. The Mean Value Inequality imply that
there exists $q\in(\frac{k}{n}+\frac{\zeta}{5},\frac{k}{n}+\frac{2\zeta}{5}),$
such that 
\[
\frac{d\mu_{q}\left(\f^{\uparrow}\right)}{dq}\le\frac{\mu_{\frac{k}{n}+\frac{2\zeta}{5}}(\f^{\uparrow})-\mu_{\frac{k}{n}+\frac{\zeta}{5}}(\f^{\uparrow})}{\left(\zeta/5\right)}\le\frac{5}{\zeta},
\]
By Friedgut's junta Theorem, there exists a set $J\subseteq\left[n\right]$
with $|J|=O_{\zeta}(1)$ and a family $\mathcal{G}\subseteq\mathcal{P}(J)$,
such that $\mu_{q}\left(\f^{\uparrow}\Delta\left\langle \g\right\rangle \right)<\frac{\epsilon^{2}}{16}$. 
\begin{claim}
$\mu_{q}\left((\f^{\uparrow})_{J}^{\varnothing}\right)<\frac{\epsilon}{4}.$ 
\end{claim}

\begin{proof}
Note that the the fact that $\f^{\uparrow}$ is monotone implies that
$\mu_{q}\left(\left(\f^{\uparrow}\right)_{J}^{B}\right)\ge\mu_{q}\left(\left(\f^{\uparrow}\right)_{J}^{\varnothing}\right).$
Suppose for a contradiction that $\mu_{q}\left((\f^{\uparrow})_{J}^{\varnothing}\right)\ge\frac{\epsilon}{4}$.
Then 
\begin{align*}
\mu_{q}\left(\f^{\uparrow}\backslash\langle\g\rangle\right) & =\sum_{B\notin\g}q^{\left|B\right|}\left(1-q\right)^{\left|J\right|-\left|B\right|}\mu_{q}\left((\f^{\uparrow})_{J}^{B}\right)\\
 & \ge\sum_{B\notin\g}q^{\left|B\right|}\left(1-q\right)^{\left|J\right|-\left|B\right|}\frac{\epsilon}{4}\ge\frac{\epsilon}{4}\left(1-\mu_{q}\left(\langle\g\rangle\right)\right)\\
 & \ge\frac{\epsilon}{4}\left(1-\mu_{q}\left(\f^{\uparrow}\right)-\mu_{q}\left(\langle\g\rangle\backslash(\f^{\uparrow})\right)\right)\\
 & \ge\frac{\epsilon}{4}\left(\frac{\epsilon}{2}-\frac{\epsilon^{2}}{16}\right)>\frac{\epsilon^{2}}{16}
\end{align*}
a contradiction.
\end{proof}
Note that $\left(\f_{J}^{\varnothing}\right)^{\uparrow}=\left(\f^{\uparrow}\right)_{J}^{\varnothing}$.
Since $\mu\left(\left(\f_{J}^{\varnothing}\right)^{\uparrow r}\right)\ge\mu\left(\left(\f_{J}^{\varnothing}\right)\right)$
for any $r\ge k$, we have 

\begin{align*}
\mu_{q}\left(\left(\f_{J}^{\varnothing}\right)^{\uparrow}\right) & =\mathbb{E}_{\boldsymbol{r}\sim\mathrm{Bin}\left(n,q\right)}\left[\Pr_{\boldsymbol{A}\sim\mu_{p}}\left[\boldsymbol{A}\in\left(\f_{J}^{\varnothing}\right)^{\uparrow}|\,\left|\boldsymbol{A}\right|=\boldsymbol{r}\right]\right]\\
 & =\mathbb{E}_{\boldsymbol{r}\sim\mathrm{Bin}\left(n,q\right)}\left[\mu\left(\left(\f_{J}^{\varnothing}\right)^{\uparrow\boldsymbol{r}}\right)\right]\\
 & \ge\mathbb{E}_{\boldsymbol{r}\sim\mathrm{Bin}\left(n,q\right)}\left[\mu\left(\left(\f_{J}^{\varnothing}\right)\right)\right]\Pr_{\boldsymbol{r}\sim\mathrm{Bin}\left(n,q\right)}\left[r\ge k\right].\\
 & \ge\mu\left(\f_{J}^{\varnothing}\right)\frac{1}{2},
\end{align*}
 provided that $n$ is sufficiently large. Rearranging and using the
claim, we obtain 
\[
\mu\left(\f_{J}^{\varnothing}\right)\le2\mu_{q}\left(\left(\f_{J}^{\varnothing}\right)^{\uparrow}\right)\le\frac{\epsilon}{2}.
\]
 However, if $\delta$ is sufficiently small to satisfy $\left|J\right|<\left\lceil \frac{1}{\delta}\right\rceil $
and $\text{\ensuremath{\delta}}<\frac{\epsilon}{2}$, then we obtain
\[
\mu\left(\f_{J}^{\varnothing}\right)\ge\mu\left(\f\right)-\delta>\frac{\epsilon}{2},
\]
 a contradiction.
\end{proof}

\subsection{Showing that every family that does not contain a $\left(d,k,\left(\frac{d+1}{d}+\zeta\right)k\right)$-cluster
is essentially contained in a $\left(d+1\right)$-wise intersecting
$j$-junta }
\begin{prop}
\label{prop:Junta approximation theorem } For each $\epsilon,\zeta>0$
there exists a $j>0,n_{0}\in\mathbb{N}$, such that the following
holds. Let $n>n_{0}$, let $\zeta n<k<\left(\frac{d}{d+1}-\zeta\right)n$,
and let $\f\subseteq\binom{\left[n\right]}{k}$ be some family that
does not contain a $\left(d,k,\left(\frac{d+1}{d}+\zeta\right)k\right)$-cluster.
Then $\f$ is $\epsilon$-essentially contained in a $\left(d+1\right)$-wise
intersecting $j$-junta.
\end{prop}

\begin{proof}
Let $\delta=\delta\left(\epsilon,\zeta,d\right)$ be sufficiently
small, and choose $j=j\left(\delta,\epsilon,\zeta,d\right),n_{0}=n_{0}\left(j,\zeta\right)$
sufficiently large. By Theorem \ref{thm:regularity} there exists
a set $J$ of size $j$ and a family $\g\subseteq\p\left(J\right)$,
such that $\f$ is $\epsilon$-essentially contained in $\left\langle \g\right\rangle ,$
and such that for each $B\in\g$ the family $\f_{J}^{B}$ is $\left(\left\lceil \frac{1}{\delta}\right\rceil ,\delta\right)$-regular
and has measure $\ge\frac{\epsilon}{2}$. If the junta $\left\langle \g\right\rangle $
is $\left(d+1\right)$-wise intersecting, then we are done. Otherwise,
there exist sets $B_{0},\ldots,B_{d}\in\g$ whose intersection is
empty. 

Let $l=\left(1+\frac{\zeta}{3}\right)k.$ Note that we may apply Lemma
\ref{lem: Follows from Friedgut-Junta Theorem} with $\frac{\zeta}{3}$
instead of $\zeta$, $\min\left\{ \frac{\epsilon}{2},\frac{1}{d+1}\right\} $
instead of $\epsilon,$ and $\delta$, provided that $\delta$ is
sufficiently small. Doing so, we obtain that $\mu\left(\left(\f_{J}^{B_{i}}\right)^{\uparrow l}\right)>1-\frac{1}{d+1}$
for each $i.$

Choose a set $\boldsymbol{S}\subseteq\left[n\right]\backslash J$
of size $\left\lceil \frac{d+1}{d}l\right\rceil $ uniformly at random.
Choose a uniformly random $\left(d,l,\left\lceil \frac{d+1}{d}l\right\rceil \right)$-cluster
$\left\{ \boldsymbol{C}_{0},\ldots,\boldsymbol{C}_{d}\right\} $ in
$\binom{\boldsymbol{S}}{l}.$ Note that each $\boldsymbol{C}_{i}$
is a uniformly random set in $\binom{\left[n\right]\backslash J}{l}$.
Therefore, 
\[
\Pr\left[\boldsymbol{C}_{i}\notin\left(\f_{J}^{B_{i}}\right)^{\uparrow l}\right]=1-\mu\left(\left(\f_{J}^{B_{i}}\right)^{\uparrow l}\right)<\frac{1}{d+1}.
\]
 A union bound implies that the probability that $\boldsymbol{C}_{i}\notin\left(\f_{J}^{B_{i}}\right)^{\uparrow l}$
for some $i$ is at most 
\[
\sum_{i=0}^{d}\left(1-\mu\left(\left(\f_{J}^{B_{i}}\right)^{\uparrow l}\right)\right)<1.
\]
 Therefore, there exists sets $C_{0}\in\left(\f_{J}^{B_{0}}\right)^{\uparrow l},\ldots,C_{d}\in\left(\f_{J}^{B_{d}}\right)^{\uparrow l}$
that form a $\left(d,l,\left\lceil \frac{d+1}{d}l\right\rceil \right)$-cluster.
By definition, this implies that there exists sets $A_{0},\ldots,A_{d}\in\f$
such that $A_{i}\subseteq B_{i}\cup C_{i}.$ Now the sets $A_{0},\ldots,A_{d}$
form a $\left(d,k,\left\lceil \frac{d+1}{d}l\right\rceil +j\right)$-cluster.
This is a contradiction since 
\[
\left\lceil \frac{d+1}{d}l\right\rceil +j=\left\lceil \frac{d+1}{d}\left(1+\frac{\zeta}{3}\right)k\right\rceil +j\le\left(\frac{d+1}{d}+\zeta\right)k,
\]
 provided that $n_{0}$ is sufficiently large. 
\end{proof}

\subsection{Proof of Proposition \ref{prop:Rough stability result}}

We shall need the following stability result for Frankl's Theorem
that was essentially proved by Ellis, Keller, and the author \cite{ellis2016stability}.
We shall use the following corollary of their work stated by Keller
and the author in \cite{keller2017junta}.
\begin{thm}[{\cite[Proposition 10.11]{keller2017junta}}]
\label{thm:Stability to Frankl's result} For each $\zeta,s>0$,
there exists $C>0$, such that the following holds. Let $\zeta<\frac{k}{n}<\frac{s-1}{s}-\zeta$,
let $\epsilon>0$, and let $\f\subseteq\binom{\left[n\right]}{k}$
be an $s$-wise intersecting family. If $\left|\f\right|\ge\binom{n-1}{k-1}\left(1-\epsilon\right),$
then $\f$ is $C\epsilon^{1+\frac{1}{C}}$-essentially contained in
a star.
\end{thm}

\begin{proof}[Proof of Proposition \ref{prop:Rough stability result}]
 Since the theorem becomes stronger as $\epsilon$ decreases, we
may assume that $\epsilon$ is sufficiently small as a function of
$\zeta,d$. Let $\f\subseteq\binom{\left[n\right]}{k}$ be a family
that does not contain a $\left(d,k,\left(\frac{d+1}{d}+\zeta\right)k\right)$-cluster
and suppose that $\left|\f\right|\ge\binom{n-1}{k-1}\left(1-\delta\right).$
By Proposition \ref{prop:Junta approximation theorem }, there exists
some $\left(d+1\right)$-wise intersecting family $\j$, such that
$\f$ is $\frac{\epsilon}{2}$-essentially contained in $\j.$ This
implies that 
\[
\left|\j\right|\ge\left|\f\right|-\frac{\epsilon}{2}\binom{n-1}{k-1}\ge\left(1-\delta-\frac{\epsilon}{2\zeta}\right)\binom{n-1}{k-1}.
\]
 By Theorem \ref{thm:Stability to Frankl's result}, the family $\j$
is $C\left(\delta+\frac{\epsilon}{2\zeta}\right)^{C}$-essentially
contained in a star, where $C=C\left(\zeta,d\right)>1$. Therefore,
$\f$ is $\frac{\epsilon}{2}+C\left(\delta+\frac{\epsilon}{2\zeta}\right)^{C}$-essentially
contained in a star. This completes the proof since $\frac{\epsilon}{2}+C\left(\delta+\frac{\epsilon}{2\zeta}\right)^{C}<\epsilon$,
provided that $\delta,\epsilon$ are sufficiently small.
\end{proof}

\section{\label{sec:exact result}The bootstrapping step: Proof of theorems
\ref{thm:exact cluster result} and \ref{thm:stability cluster result} }
\begin{proof}[Proof of Theorem \ref{thm:stability cluster result}]
 Note that by increasing $C$ if necessary we may assume that $\epsilon$
is sufficiently small. Let $\epsilon'=\epsilon'\left(\zeta,d\right)$
be sufficiently small. By Proposition (\ref{prop:Rough stability result}),
the family $\f$ is $\epsilon'$-essentially contained in a star,
provided that $\epsilon$ is sufficiently small. Without loss of generality,
it is the star $\mathcal{S}$ of all sets that contain 1. Let $l=\left\lceil k\left(1+\frac{\zeta}{2}\right)\right\rceil .$
Similarly to the proof of Proposition \ref{prop:Junta approximation theorem },
the families 
\[
\left(\f_{\left\{ 1\right\} }^{\left\{ 1\right\} }\right)^{\uparrow l},\ldots,\left(\f_{\left\{ 1\right\} }^{\left\{ 1\right\} }\right)^{\uparrow l},\left(\f_{\left\{ 1\right\} }^{\varnothing}\right)^{\uparrow l}
\]
 do not mutually contain a $\left(d,l,\left\lceil \frac{d+1}{d}l\right\rceil \right)$-cluster.
Let $\left\{ \boldsymbol{A}_{0},\ldots,\boldsymbol{A}_{d}\right\} \subseteq\binom{\left[n\right]\backslash\left\{ 1\right\} }{l}$
be a uniformly random $\left(d,l,\left\lceil \frac{d+1}{d}l\right\rceil \right)$-cluster.
We have
\begin{align}
1 & =\Pr\left[\exists i:\,\boldsymbol{A}_{i}\notin\f\right]\le\sum_{i=0}^{d}\Pr\left[\boldsymbol{A}_{i}\notin\f\right]\nonumber \\
 & =d\left(1-\mu\left(\left(\f_{\left\{ 1\right\} }^{\left\{ 1\right\} }\right)\right)\right)+\left(1-\mu\left(\f_{\left\{ 1\right\} }^{\varnothing}\right)\right).\label{eq:union bound}
\end{align}

Write $\left|\f\cap\mathcal{S}\right|=\binom{n-1}{k-1}\left(1-\epsilon''\right).$
By Corollary \ref{cor:kk up},
\[
\left|\left(\f\cap\mathcal{S}\right)^{\uparrow l}\right|\ge\binom{n-1}{l-1}\left(1-C'\epsilon''{}^{\frac{1}{C'}}\right),
\]
where $C'=C'\left(\zeta,d\right)>1.$ Hence, $\mu\left(\left(\f_{\left\{ 1\right\} }^{\left\{ 1\right\} }\right)^{\uparrow l}\right)\ge1-C'\epsilon''{}^{\frac{1}{C'}}.$
By (\ref{eq:union bound}), 
\[
\mu\left(\left(\f_{\left\{ 1\right\} }^{\varnothing}\right)^{\uparrow l}\right)\le d\left(1-\mu\left(\left(\f_{\left\{ 1\right\} }^{\left\{ 1\right\} }\right)^{\uparrow l}\right)\right)\le dC'\epsilon''^{\frac{1}{C'}}.
\]
 Therefore, 
\begin{equation}
\mu\left(\f\backslash\mathcal{S}\right)\le\mu\left(\f_{\left\{ 1\right\} }^{\varnothing}\right)\le\mu\left(\left(\f_{\left\{ 1\right\} }^{\varnothing}\right)^{\uparrow l}\right)\le dC'\epsilon''^{1+\frac{1}{C''}}.\label{eq:fmins}
\end{equation}
Hence, 
\begin{align*}
\frac{k}{n}\left(1-\epsilon\right) & \le\mu\left(\f\right)=\mu\left(\f\cap\mathcal{S}\right)+\mu\left(\f\backslash\mathcal{S}\right)\\
 & \le\frac{k}{n}\left(1-\epsilon''\right)+\min\left\{ \left(1-\frac{k}{n}\right)dC'\epsilon''^{1+\frac{1}{C'}},\epsilon'\right\} .
\end{align*}
 Rearranging, we obtain 
\begin{equation}
\epsilon''\le\epsilon+\min\left\{ \frac{\left(1-\frac{k}{n}\right)}{\frac{k}{n}}dC'\epsilon''^{1+\frac{1}{C'}},\frac{n}{k}\epsilon'\right\} .\label{eq:primeprime}
\end{equation}
 In particular, $\epsilon''\le\epsilon+\frac{1}{\zeta}\epsilon'$.
Since both $\epsilon,\epsilon'$ may be assumed to be arbitrarily
small as a function of $d,\zeta,C'$, we may assume that $\epsilon''$
is sufficiently small to have 
\[
\frac{\left(1-\frac{k}{n}\right)}{\frac{k}{n}}dC'\epsilon''^{1+\frac{1}{C'}}\le\frac{\epsilon''}{2}.
\]
 Combining with (\ref{eq:primeprime}), we have $\epsilon''\le2\epsilon.$
Hence, by (\ref{eq:fmins})
\[
\mu\left(\f\backslash\mathcal{S}\right)\le dC'\epsilon''^{1+\frac{1}{C''}}=\Theta_{d,\zeta}\left(\epsilon\right)^{1+\frac{1}{C'}}\le C\epsilon^{1+\frac{1}{C}},
\]
 provided that $C$ is sufficiently large. This completes the proof
of the theorem. 
\end{proof}
Theorem \ref{thm:exact cluster result} now follows easily.
\begin{proof}[Proof of Theorem \ref{thm:exact cluster result}]
 The Theorem immediately follows from Theorem \ref{thm:exact cluster result}
if $\frac{k}{n}\le\frac{d}{d+1}-\zeta'$ for any fixed $\zeta'=\zeta'\left(d,\epsilon\right)>0$,
by substituting $\epsilon=0$. On the other hand, it follows from
Theorem \ref{thm:Frankl's Theorem} if $n\le\left(\frac{d+1}{d}+\zeta\right)k.$
Hence, the theorem follows by substituting $\zeta'=\frac{d}{d+1}-\frac{1}{\left(\frac{d+1}{d}+\zeta\right)}$.
\end{proof}

\section{Open problem}

We conjecture that Theorem \ref{thm:exact cluster result} in fact
holds for all $k\ge C\log n$ for a sufficiently large constant $C.$ 
\begin{conjecture}
\label{conj:exact result}For each $d,\epsilon>0$ there exists a
constant $C=C\left(d,\epsilon\right)$, such that the following holds.
Let \textup{$C\log n\le k\le\frac{d+1}{d}n,$}\textup{\emph{ }}\textup{and
let $\mathcal{F}\subseteq\binom{\left[n\right]}{k}$ be a family that
doe}s not\textup{ contain a $\left(d,k,\left(\frac{d+1}{d}+\zeta\right)k\right)$-cluster.
Then $\left|\mathcal{F}\right|\le\binom{n-1}{k-1}$, with equality
only if and only if $\f$ is a star. }
\end{conjecture}

We would also like to mention that one cannot strengthen Conjecture
\ref{conj:exact result} by replacing the requirement that $\f$ does
not contain a $\left(d,k,\left(\frac{d+1}{d}+\zeta\right)k\right)$-cluster
by the requirement that $\f$ does not contain a $\left(d,k,\frac{d+1}{d}k\right)$-cluster.
For instance, in the case where $d=2$ one can take the following
example known as the complete \emph{odd bipartite hypergraph}. 
\begin{example}
\label{exa: Keevash counter }The family $\f:=\left\{ A\in\binom{\left[n\right]}{k}|\,\left|A\cap\left[\frac{n}{2}\right]\right|\text{ is odd}\right\} $
does not contain a $\left(2,k,\frac{3}{2}k\right)$-cluster and its
measure is asymptotically $\frac{1}{2}.$ Therefore, for any $\zeta>0$
and $\frac{k}{n}<\frac{1}{2}-\zeta$, we have $\left|\f\right|\ge\binom{n-1}{k-1}$,
provided that $n\ge n_{0}\left(\zeta\right).$ 
\end{example}

\subsection*{Acknowledgment}

I would like to thank Nathan Keller and Zoltán Füredi for many helpful
discussions.

\end{document}